\documentclass[11pt]{amsart}

\usepackage[T1]{fontenc}
\usepackage{lmodern}
\usepackage{amsmath,amssymb,amsthm,mathtools}
\usepackage{microtype}
\usepackage{xcolor}
\usepackage[colorlinks=true,linkcolor=blue!55!black,citecolor=green!45!black,
            urlcolor=blue!65!black]{hyperref}

\numberwithin{equation}{section}
\allowdisplaybreaks

\newtheorem{theorem}{Theorem}[section]

\newtheorem{lemma}[theorem]{Lemma}
\newtheorem{corollary}[theorem]{Corollary}

\theoremstyle{remark}
\newtheorem{remark}[theorem]{Remark}

\title[Quantized volume comparison]{Quantized Volume Comparison\\
for Fano Manifolds, II}

\author{Kaixuan Lyu and Kewei Zhang}
\address{School of Mathematical Sciences, Beijing Normal University, Beijing 100871, China}
\email{17852841271@163.com, kwzhang@bnu.edu.cn}

\begin{document}

\begin{abstract}
In this note, the second author's quantized volume comparison conjecture is solved: If $X$ is a $K$-semistable Fano manifold of dimension $n$, then for every
integer $m\geq1$,
\[
 h^0(X,-mK_X)\leq
 h^0(\mathbb P^n,-mK_{\mathbb P^n})
 =\binom{n+m(n+1)}{n},
\]
and equality for one $m$ characterizes projective space. Somewhat surprisingly, the same statement actually holds whenever $T_X$ is
slope semistable with respect to $-K_X$.  The idea is to apply a jet-dimension counting trick to a filtration of subsheaves induced by $H^0(X,-mK_X)$. Then the slope semistability condition yields the desired dimension bound.
\end{abstract}

\maketitle

\section{Introduction}
\label{sec:introduction}

Let $X$ be a complex Fano manifold of dimension $n$.  A theorem of
K. Fujita \cite{Fujita2018} says that if $X$ is $K$-semistable, then
\[
   (-K_X)^n\leq (n+1)^n=(-K_{\mathbb P^n})^n,
\]
and equality holds if and only if $X\simeq \mathbb P^n$.  Since
asymptotic Riemann--Roch gives
\[
 h^0(X,-mK_X)=\frac{(-K_X)^n}{n!}m^n+O(m^{n-1}),
\]
it is natural to ask whether Fujita's comparison admits a
finite-level, or quantized, refinement.

The second author formulated precisely this question in
\cite[Conjecture~1.2]{Zhang2025}.  He conjectured that every
$K$-semistable Fano manifold satisfies
\begin{equation}
  h^0(X,-mK_X)
  \leq h^0(\mathbb P^n,-mK_{\mathbb P^n})
  =\binom{n+m(n+1)}{n}
  \qquad\text{for every }m\geq 1,
  \label{eq:quantized-comparison-intro}
\end{equation}
with equality, for one value of $m$, only for projective space. 

This expectation was actually inspired by a conjecture of Yau in the non-compact setting, stating that
\begin{equation}
    \label{eq:Yau}
    \dim_\mathbb C\mathcal{O}_d(M^n)\leq \dim_\mathbb C\mathcal{O}_d(\mathbb C^n)=\binom{n+d}{n}\ \text{for every }d>0,
\end{equation}
where $M^n$ is a complete noncompact K\"ahler manifold with nonnegative holomorphic
bisectional curvature and $\mathcal{O}_d(M)$ denotes the space of holomorphic functions on $M$ with polynomial growth of order at most $d$. Regarding this problem,
Mok had earlier obtained related, non-sharp polynomial dimension estimates
under additional curvature-decay and volume-growth assumptions
\cite{Mok1984Embedding}. Ni obtained the optimal-order estimate
$\dim_{\mathbb C}\mathcal O_d(M)\leq C(n)d^n$ for $d\geq1$ in general
and proved Yau's sharp comparison under the additional assumption of
maximal volume growth \cite{Ni2004Monotonicity}. Chen--Fu--Yin--Zhu
subsequently removed that assumption and established the full rigidity
statement: the inequality \eqref{eq:Yau} holds true and equality happens only when $M$ is biholomorphically isometric to $\mathbb{C}^n$, thereby completing Yau's original bisectional-curvature
conjecture \cite{ChenFuYinZhu2006}. Liu later gave a different proof using his three-circle theorem and
strengthened the result by replacing nonnegative holomorphic bisectional
curvature with the strictly weaker assumption of nonnegative holomorphic
sectional curvature \cite{Liu2016ThreeCircle}.

In \cite{Zhang2025}, the second author initiated the study of the analogous problem in the compact setting, where he proved the comparison \eqref{eq:quantized-comparison-intro} uniformly for all sufficiently large $m$ and
proved a conditional reduction
in dimension four \cite[Theorem~1.1 and
Proposition~2.4]{Zhang2025}.

In this note we provide a complete solution to this problem.

\begin{theorem}\label{thm:main-intro}
Let $X$ be an $n$-dimensional smooth complex $K$-semistable Fano
variety.  Then, for every integer $m\geq 1$,
\[
  h^0(X,-mK_X)\leq \binom{n+m(n+1)}{n}.
\]
If equality holds for some $m\geq1$, then $X\simeq\mathbb P^n$.
\end{theorem}

In fact, the same result actually holds under the weaker assumption that $T_X$ is slope-semistable. 
The slope stability of tangent bundles of Fano manifolds has a long and
largely independent history.  Tian's work \cite{Tian1992} is the analytic
starting point for the relation between K\"ahler--Einstein geometry and
stability of the tangent bundle.  On the algebraic side,
Peternell--Wi\'sniewski studied Fano manifolds of Picard number one and high
index \cite{PeternellWisniewski1995}, while Steffens analyzed the tangent
bundles of Fano threefolds and related destabilizing subsheaves to extremal
contractions \cite{Steffens1996}. A particularly close predecessor for our main result is Ran's
work: using slope estimates for
sheaves of differential operators, together with some additional assumptions on $X$, he proved the non-sharp bound
\[
  (-K_X)^n\leq(2n)^n
\]
when the tangent sheaf is slope semistable \cite{Ran2002}.  More recent work
includes combinatorial criteria and classifications for smooth toric varieties
\cite{BiswasDeyGencPoddar2021,HeringNillSuess2022}.  Kanemitsu found
Picard-number-one Fano manifolds with unstable tangent bundle, disproving the
formerly expected general stability statement \cite{Kanemitsu2021}.  These
works are principally concerned with deciding when $T_X$ is (semi)stable and
with describing its destabilizing subsheaves.  To the best of our knowledge,
the sharp anticanonical Hilbert-function and volume estimates below have not
previously been recorded under the sole hypothesis that $T_X$ is
$(-K_X)$-slope semistable.  They should be compared with the Fujita--Liu
volume comparison under the stronger hypothesis of $K$-semistability and
with the Tian--Li implication from $K$-semistability to tangent-sheaf
semistability.

Here is the stronger form of
Theorem~\ref{thm:main-intro}.

\begin{theorem}
\label{thm:tangent-semistable-intro}
Let $X$ be a smooth complex Fano manifold of dimension $n$, and assume that
$T_X$ is slope semistable with respect to $H=-K_X$.  Then, for every integer
$m\geq1$,
\[
 h^0(X,-mK_X)
 \leq h^0(\mathbb P^n,-mK_{\mathbb P^n})
 =\binom{n+m(n+1)}{n}.
\]
Consequently,
\[
 (-K_X)^n\leq(n+1)^n.
\]
If equality in the finite-level dimension estimate holds for one
$m\geq1$, or $(-K_X)^n=(n+1)^n$ holds, then $X\simeq\mathbb P^n$.
\end{theorem}

In fact, the anticanonical statement is a specialization of a result for a general
polarization and a general line bundle.  Namely, let $(X,A)$ be a smooth
polarized $n$-fold such that $T_X$ is $\mu_A$-semistable and
\[
 \tau_A(X):=(-K_X)\cdot A^{n-1}>0.
\]
For every line bundle $L$ and every $m\geq1$ with $h^0(X,mL)>0$, we prove
in Theorem~\ref{thm:general-line-bundle-bound} that
\[
 h^0(X,mL)\leq
 \binom{n+
 \left\lceil (n+1)m\,
 \dfrac{L\cdot A^{n-1}}{(-K_X)\cdot A^{n-1}}
 \right\rceil}{n}.
\]
If $L$ is big, Corollary~\ref{cor:general-line-bundle-fano-comparison}
gives the corresponding volume estimate
\[
 \operatorname{vol}(L)\leq(n+1)^n
 \left(
 \frac{L\cdot A^{n-1}}{(-K_X)\cdot A^{n-1}}
 \right)^n.
\]
Thus Theorem~\ref{thm:tangent-semistable-intro} is obtained by taking
$A=L=-K_X$.

Our proof uses slope semistability of the tangent bundle rather than the
valuative criterion directly \cite{Fujita2019Valuative,Li2017Valuative}.  Its analytic ancestor is Tian's theorem on
the canonical extension of the tangent bundle under K\"ahler--Einstein and
Ricci-lower-bound hypotheses \cite[Theorems~0.1--0.2]{Tian1992}.  The exact
modern implication used here is Li's extension of Tian's result:
\cite[Theorem~1.3]{Li2021} proves that $K$-semistability of a smooth Fano
manifold implies that $T_X$ is slope semistable with respect to
$H:=-K_X$.
Fix $m\geq1$, put $L=mH$, and set
\[
  W=H^0(X,L),\qquad N=\dim W.
\]
Let $\mathcal P^j_X(L)$ be the sheaf of principal parts of order $j$.
The kernels of the jet-evaluation maps
\[
  W\otimes\mathcal O_X\longrightarrow \mathcal P^j_X(L)
\]
give a finite decreasing filtration.  Its $j$-th graded image
$Q_j$ is a torsion-free subsheaf of
\[
  \operatorname{Sym}^j\Omega_X^1\otimes L.
\]
These symmetric powers are slope semistable as well. Moreover,
\[
 \mu_H\!\left(\operatorname{Sym}^j\Omega_X^1\otimes L\right)
   =H^n\left(m-\frac{j}{n}\right).
\]
Summing the resulting slope inequalities and using the telescoping of
the first Chern classes of the jet filtration gives, for a general
point $p\in X$,
\begin{equation}
  \sum_{a\geq1}h^0(X,L\otimes\mathfrak m_p^a)
     \leq mnN.
  \label{eq:finite-expected-vanishing-intro}
\end{equation}

The rest of the argument uses ideas from our previous work \cite{Zhang2025}. The space of
homogeneous Taylor coefficients of order $j$ in $n$ variables has
dimension
\[
   d_j=\binom{n+j-1}{n-1}.
\]
Writing $q=m(n+1)$, one has
\[
 \sum_{j=0}^{q}d_j=\binom{n+q}{n},
 \qquad
 \sum_{j=0}^{q}j\,d_j=mn\binom{n+q}{n}.
\]
The rank of the $j$-th graded piece of the jet filtration is at most
$d_j$.  Consequently, if $N>\binom{n+q}{n}$, even the smallest possible
weighted sum of the jet orders is strictly larger than $mnN$, in
contradiction with \eqref{eq:finite-expected-vanishing-intro}.  This
proves \eqref{eq:quantized-comparison-intro}.

The equality case is also handled as in \cite{Zhang2025}: we have that
\[
 H^0(X,L)\longrightarrow
 L_p\otimes\mathcal O_{X,p}/\mathfrak m_p^{q+1}
\]
is an isomorphism, so $L$ generates $q=m(n+1)$ jets at $p$.  It follows
that $\varepsilon(-K_X,p)\geq n+1$.  The characterization of projective
space by the anticanonical Seshadri constant
\cite[Theorem~1.7]{BauerSzemberg2009} then gives
$X\simeq\mathbb P^n$.

In \cite{Zhang2025} we proved the same numerical comparison at a fixed level $m$
under the finite-level hypothesis $\delta_m(-K_X)\geq1$
\cite[Theorem~1.3]{Zhang2025}.  Our argument here does not attempt to deduce this
hypothesis from $K$-semistability: it controls the ordinary vanishing-order
filtration at a general point, whereas $\delta_m$ is an infimum over all
valuations.

One should compare our work with Liu's solution to Yau's conjecture \eqref{eq:Yau}. Taking $d=m(n+1)$, the Euclidean model dimension becomes
\[
 \dim_{\mathbb C}\mathcal O_{m(n+1)}(\mathbb C^n)
 =\binom{n+m(n+1)}{n}
 =h^0(\mathbb P^n,-mK_{\mathbb P^n}),
\]
which is exactly the comparison term in our theorem. Thus both results are
sharp jet-counting statements with model-space rigidity. Liu's three-circle
theorem gives
$\operatorname{ord}_p(f)\leq\lfloor d\rfloor$ for every nonzero
$f\in\mathcal O_d(M)$, whereas our argument bounds the jet-dimension of $H^0(X,-mK_X)$ through slope semistability. The analogy lies in this common
local-to-global counting mechanism; the hypotheses themselves are of a
different nature, being a pointwise curvature condition on a complete
noncompact manifold in Liu's theorem and an algebraic slope-semistability
condition on a compact Fano manifold in ours.

The paper is organized as follows.  Section~\ref{sec:jet-filtration} proves a
general jet-filtration theorem under slope semistability of $T_X$ and derives
the main result in the smooth case.

\textbf{Acknowledgments.} K. Zhang is supported by Scientific Research
Innovation Capability Support Project for Young Faculty SRICSPYF-ZY2025169, and by NSFC grants 12571060, 12271038 and 12271040.
We are grateful to Hanfeng Chu for providing the proof of a missing equality case in the previous version of Theorem \ref{thm:tangent-semistable-intro}.

\subsection*{Declaration on the Use of AI}
The argument of this work has two main ingredients. The first one is a sharp jet counting estimate of holomorphic sections (Lemma \ref{lem:taylor-capacity}), whose equality case implies a lower bound of the Seshadri constant. The second one is a novel jet-filtration slope bound (Lemma \ref{lem:jet-filtration}), where the semistability of vector bundles enters crucially. The first ingredient already appeared in our earlier work \cite{Zhang2025} (which builds on an idea from K. fujita \cite{Fujita2018}); this point has also been applied in the recent resolution of Ehrhart's volume conjecture \cite{OpenAI2026,Liu2026Ehrhart}.
But the second ingredient here is due to ChatGPT 5.6, which probably borrows ideas from the early work of Ran \cite{Ran2002}, and the authors do not claim any credit on this point. ChatGPT was also used to assist with writing and exposition. The authors take full responsibility for the content of this note.

\section{A jet-filtration bound from slope semistability}
\label{sec:jet-filtration}

Throughout this section, all varieties are defined over $\mathbb C$.  If $A$ is
an ample line bundle on a smooth projective $n$-fold and $\mathcal E$ is a
torsion-free coherent sheaf of positive rank, we use the conventions
\[
  \deg_A(\mathcal E):=c_1(\mathcal E)\cdot A^{n-1},
  \qquad
  \mu_A(\mathcal E):=\frac{\deg_A(\mathcal E)}{\operatorname{rk}(\mathcal E)}.
\]
Here $c_1(\mathcal{E}):=c_1( \det\mathcal E)=c_1((\bigwedge^{\mathrm{rk}(\mathcal{E)}}\mathcal E)^{**}).$
We say that $\mathcal E$ is $\mu_A$-semistable if every coherent subsheaf
$0\neq\mathcal F\subsetneq\mathcal E$ with
$0<\operatorname{rk}\mathcal F<\operatorname{rk}\mathcal E$ satisfies
\[
  \mu_A(\mathcal F)\leq\mu_A(\mathcal E).
\]
Equivalently, it is enough to test saturated subsheaves.  We shall also use
the dual quotient criterion: $\mathcal E$ is $\mu_A$-semistable if and only if
every nonzero torsion-free quotient $\mathcal E\twoheadrightarrow\mathcal G$
satisfies $\mu_A(\mathcal G)\geq\mu_A(\mathcal E)$.  These standard equivalent
forms are recalled, for example, in \cite[Section~1.2]{HuybrechtsLehn2010}.

The following statement is slightly stronger than the quantized volume
comparison conjecture \eqref{eq:quantized-comparison-intro}, since it only uses slope semistability of the tangent
bundle.

\begin{theorem}[=Theorem \ref{thm:tangent-semistable-intro}]\label{thm:semistable-tangent-quantized-bound}
  Let $X$ be a smooth Fano variety of dimension $n\geq 1$, and put
  $H=-K_X$.  Suppose that $T_X$ is $\mu_H$-semistable.  Then, for every
  integer $m\geq 1$,
  \[
    h^0(X,mH)
    \leq
    \binom{n+m(n+1)}{n}
    =h^0\bigl(\mathbb P^n,-mK_{\mathbb P^n}\bigr).
  \]
  If equality holds for one integer $m\geq 1$, then $X\simeq\mathbb P^n$.
  If $(-K_X)^n=(-K_{\mathbb P^n})^n$ holds, then one also has $X\simeq\mathbb P^n$.
\end{theorem}

We first record the semistability facts used below.

\begin{lemma}\label{lem:symmetric-powers-semistable}
  Let $(X,A)$ be a smooth polarized projective variety over $\mathbb C$, and
  let $\mathcal E$ be a $\mu_A$-semistable vector bundle.  Then
  $\mathcal E^\vee$ is $\mu_A$-semistable, and
  $\operatorname{Sym}^j\mathcal E$ is $\mu_A$-semistable for every $j\geq 0$.
  Moreover,
  \[
    \mu_A\bigl(\operatorname{Sym}^j\mathcal E\bigr)
    =j\mu_A(\mathcal E).
  \]
\end{lemma}

\begin{proof}
  The assertion for duals follows from the quotient criterion.  If
  $\mathcal G\subset\mathcal E^\vee$ is a nonzero saturated subsheaf, then
  dualizing $\mathcal G\hookrightarrow\mathcal E^\vee$ gives a generically
  surjective map $\mathcal E\to\mathcal G^\vee$.  Its torsion-free image
  $\mathcal I$ is a quotient of $\mathcal E$ and agrees with $\mathcal G^\vee$
  in codimension one.  Hence
  $\mu_A(\mathcal I)=\mu_A(\mathcal G^\vee)=-\mu_A(\mathcal G)$.  Since
  $\mathcal E$ is semistable, the quotient criterion gives
  $-\mu_A(\mathcal G)\geq\mu_A(\mathcal E)$, or
  $\mu_A(\mathcal G)\leq\mu_A(\mathcal E^\vee)$.

  We spell out the standard characteristic-zero argument.  Suppose that
  $\mathcal E^{\otimes j}$ is not $\mu_A$-semistable, and let
  $\mathcal F\subset\mathcal E^{\otimes j}$ be a saturated destabilizing
  subsheaf.  If $\dim X=1$, put $C=X$; otherwise, for $k\gg0$, choose a
  general smooth complete-intersection curve
  \[
    C=H_1\cap\cdots\cap H_{n-1},
    \qquad H_i\in |kA|,
  \]
  which avoids the codimension-two loci where $\mathcal F$ or its quotient is
  not locally free.  The Mehta--Ramanathan restriction theorem implies that
  $\mathcal E|_C$ is semistable.  Hence $(\mathcal E|_C)^{\otimes j}$ is
  semistable, since tensor products of semistable bundles on a smooth complex
  curve are semistable.  On the other hand,
  $\mathcal F|_C\subset(\mathcal E|_C)^{\otimes j}$ remains destabilizing,
  because all slopes are multiplied by $k^{n-1}$.  This contradiction proves
  that $\mathcal E^{\otimes j}$ is semistable.  The symmetrizing idempotent
  \[
    \frac{1}{j!}\sum_{\sigma\in\mathfrak S_j}\sigma
  \]
  exhibits $\operatorname{Sym}^j\mathcal E$ as a direct summand of
  $\mathcal E^{\otimes j}$.  Its slope is $j\mu_A(\mathcal E)$ by the splitting
  principle, which is also the slope of $\mathcal E^{\otimes j}$.  A
  direct summand of the same slope in a semistable bundle is semistable.  See
  \cite{MehtaRamanathan1982,RamananRamanathan1984,Totaro1994} for the
  restriction, tensor-product, and tensor-construction results used here.
\end{proof}

For a line bundle $L$ and a finite-dimensional vector subspace
$W\subseteq H^0(X,L)$, set
\[
  F_p^a(W):=W\cap H^0(X,L\otimes\mathfrak m_p^a),
  \qquad
  T_p(W):=\sum_{a\geq 1}\dim F_p^a(W).
\]
The sum defining $T_p(W)$ is finite: a nonzero section has finite vanishing
order at $p$, and $W$ is finite-dimensional.

\begin{lemma}\label{lem:jet-filtration}
  Let $(X,A)$ be a smooth polarized irreducible projective $n$-fold.  Assume
  that $\Omega_X^1$ is $\mu_A$-semistable and
  $\mu_A(\Omega_X^1)<0$.  Let $L$ be a line bundle and let
  $W\subseteq H^0(X,L)$ be a vector subspace of dimension $N>0$.  Then there
  is a nonempty Zariski-open subset $U=U(L,W)\subseteq X$ such that, for every
  $p\in U$,
  \begin{equation}\label{eq:general-jet-bound}
    T_p(W)
    \leq
    \frac{\deg_A(L)}{-\mu_A(\Omega_X^1)}\,N.
  \end{equation}
\end{lemma}

\begin{proof}
  Let $\mathcal P^j(L)$ denote the bundle of principal parts of order $j$.
  Since $X$ is smooth, these are vector bundles and fit into exact sequences
  \begin{equation}\label{eq:principal-parts}
    0\longrightarrow
    \operatorname{Sym}^j\Omega_X^1\otimes L
    \longrightarrow \mathcal P^j(L)
    \longrightarrow \mathcal P^{j-1}(L)
    \longrightarrow 0
    \qquad (j\geq 1).
  \end{equation}
  Although the universal jet map $L\to\mathcal P^j(L)$ is a differential
  operator, every $s\in W$ determines a global section $j^j(s)$ of
  $\mathcal P^j(L)$.  Extending the map $s\mapsto j^j(s)$ by
  $\mathcal O_X$-linearity defines the evaluation morphism (between sheaves)
  \[
    e_j:W\otimes_{\mathbb C}\mathcal O_X\longrightarrow\mathcal P^j(L).
  \]
  Define
  \[
    \mathcal K_0:=W\otimes_{\mathbb C}\mathcal O_X,
    \qquad
    \mathcal K_j:=\ker(e_{j-1})\quad (j\geq 1).
  \]
  We emphasize the zeroth graded piece.  Put
  \[
    \mathcal Q_0:=\operatorname{im}(e_0)\subseteq L
      =\operatorname{Sym}^0\Omega_X^1\otimes L.
  \]
  For $j\geq 1$, the compatibility of $e_j$ with the truncation map in
  \eqref{eq:principal-parts} shows that $e_j|_{\mathcal K_j}$ factors through
  $\operatorname{Sym}^j\Omega_X^1\otimes L$; define $\mathcal Q_j$ to be its
  image.  Its kernel is precisely $\mathcal K_{j+1}$.  Thus, for every
  $j\geq 0$, there is an exact sequence of coherent sheaves
  \begin{equation}\label{eq:KQ-filtration}
    0\longrightarrow\mathcal K_{j+1}
    \longrightarrow\mathcal K_j
    \longrightarrow\mathcal Q_j
    \longrightarrow0,
    \qquad
    \mathcal Q_j\subseteq
    \operatorname{Sym}^j\Omega_X^1\otimes L.
  \end{equation}
  Here every $\mathcal K_j$ is torsion-free because it is a subsheaf of the
  locally free sheaf $\mathcal K_0=W\otimes_{\mathbb C}\mathcal O_X$ on the
  integral smooth variety $X$.  Every $\mathcal Q_j$ is torsion-free because it
  is, by construction, the image of $\mathcal K_j$ inside the vector bundle
  $\operatorname{Sym}^j\Omega_X^1\otimes L$, hence a subsheaf of a torsion-free
  sheaf.

  This filtration terminates.  Fix a closed point $x\in X$.  For every
  $j\geq1$, the canonical fiberwise
  identification
  \[
    \mathcal P^{j-1}(L)\otimes k(x)
    \simeq L_x\otimes_{\mathcal O_{X,x}}
             \mathcal O_{X,x}/\mathfrak m_x^j
  \]
  gives
  \[
    \ker\bigl(e_{j-1}(x)\bigr)=F_x^j(W).
  \]
  Moreover, $\bigcap_{j\geq0}F_x^j(W)=0$: by the Krull intersection theorem,
  a section in this intersection has zero germ at $x$, and hence vanishes on a
  nonempty open subset and therefore on the integral variety $X$.  Since $W$
  is finite-dimensional, $F_x^J(W)=0$ for some $J$.  Thus $e_{J-1}(x)$ is
  injective, so $e_{J-1}$ has generic rank $N$.  Consequently $\mathcal K_J$
  has rank zero.  As it is a torsion-free subsheaf of
  $W\otimes\mathcal O_X$, one has $\mathcal K_J=0$.  Additivity of the first
  Chern class in
  \eqref{eq:KQ-filtration} gives
  \begin{equation}\label{eq:c1-telescoping}
    \sum_{j=0}^{J-1}c_1(\mathcal Q_j)
    =c_1(\mathcal K_0)-c_1(\mathcal K_J)=0.
  \end{equation}

  Shrink to a nonempty open subset $U\subseteq X$ on which each of
  $e_0,\ldots,e_{J-1}$ has constant rank equal to its generic rank.  Constant-rank normal form
  then shows that their kernels, images, and cokernels are locally free and commute
  with taking fibres.  Hence, for $p\in U$,
  \[
    \begin{aligned}
      \mathcal K_0\otimes k(p)&=W=F_p^0(W),\\
      \mathcal K_j\otimes k(p)
      &=\ker(e_{j-1}(p))=F_p^j(W)
        \quad (1\leq j\leq J).
    \end{aligned}
  \]
  For $0\leq j\leq J-1$, consider the symbol map
  \[
    \phi_j:\mathcal K_j|_U\longrightarrow
    (\operatorname{Sym}^j\Omega_X^1\otimes L)|_U
  \]
  Then, fibrewise,
  \[
    \ker\phi_j(p)=\ker e_j(p)
      =\mathcal K_{j+1}\otimes k(p).
  \]
  Thus $\phi_j$ has constant rank
  $\operatorname{rk}(\mathcal K_j)-\operatorname{rk}(\mathcal K_{j+1})$,
  and \eqref{eq:KQ-filtration} remains exact after taking every fibre over
  $U$.
  Consequently, for $p\in U$, if
  \[
    b_j:=\operatorname{rk}(\mathcal Q_j),
  \]
  then
  \begin{equation}\label{eq:vanishing-sequence-identities}
    b_j=\dim F_p^j(W)-\dim F_p^{j+1}(W),
    \qquad
    \sum_{j=0}^{J-1}b_j=N,
    \qquad
    \sum_{j=0}^{J-1}j b_j=T_p(W).
  \end{equation}

  By Lemma~\ref{lem:symmetric-powers-semistable}, the bundle
  \[
    \mathcal E_j:=\operatorname{Sym}^j\Omega_X^1\otimes L
  \]
  is $\mu_A$-semistable and
  \[
    \mu_A(\mathcal E_j)=\deg_A(L)+j\mu_A(\Omega_X^1).
  \]
  If $0<b_j<\operatorname{rk}(\mathcal E_j)$, semistability and
  $\mathcal Q_j\subseteq\mathcal E_j$ imply
  \[
    \deg_A(\mathcal Q_j)
    \leq b_j\mu_A(\mathcal E_j).
  \]
  The same inequality holds when $b_j=\operatorname{rk}(\mathcal E_j)$.
  Set $\det\mathcal Q_j=(\bigwedge^{b_j}\mathcal Q_j)^{**}$.  The full-rank
  inclusion induces a nonzero morphism of line bundles
  $\det\mathcal Q_j\to\det\mathcal E_j$, and hence
  \[
    \det(\mathcal Q_j)\simeq\det(\mathcal E_j)(-D_j)
  \]
  for an effective divisor $D_j$, and $D_j\cdot A^{n-1}\geq0$.  If $b_j=0$,
  then $\mathcal Q_j=0$ and there is nothing to prove.  Hence in all cases
  \begin{equation}\label{eq:Qj-slope-bound}
    \deg_A(\mathcal Q_j)
    \leq
    b_j\bigl(\deg_A(L)+j\mu_A(\Omega_X^1)\bigr).
  \end{equation}
  Summing \eqref{eq:Qj-slope-bound} and using
  \eqref{eq:c1-telescoping} and
  \eqref{eq:vanishing-sequence-identities}, we obtain
  \[
    0
    \leq
    N\deg_A(L)+\mu_A(\Omega_X^1)T_p(W).
  \]
  Since $\mu_A(\Omega_X^1)<0$, this is exactly
  \eqref{eq:general-jet-bound}.
\end{proof}

We recall the following elementary jet-dimension estimate, which can be extracted from our previous work (see \cite[\S 3]{Zhang2025}).

\begin{lemma}\label{lem:taylor-capacity}
  Let $X$ be a smooth $n$-fold, let $L$ be a line bundle, and let
  $W\subseteq H^0(X,L)$ have dimension $N$.  Fix a point $p\in X$ and an
  integer $q\geq0$, and put
  \[
    D_q:=\binom{n+q}{n}.
  \]
  If $N\geq D_q$, then
  \begin{equation}\label{eq:capacity-strict}
    T_p(W)
    \geq
    (q+1)N-\binom{n+q+1}{n+1}.
  \end{equation}
  If $N=D_q$ and
  \begin{equation}\label{eq:capacity-equality-hypothesis}
    T_p(W)
    \leq
    (q+1)D_q-\binom{n+q+1}{n+1},
  \end{equation}
  then the order-$q$ jet-evaluation map
  \[
    W\longrightarrow
    L_p\otimes_{\mathcal O_{X,p}}
    \mathcal O_{X,p}/\mathfrak m_p^{q+1}
  \]
  is an isomorphism.
\end{lemma}

\begin{proof}
  Taylor expansion at the smooth point $p$ gives, for every $a\geq1$,
  \begin{equation}\label{eq:taylor-codimension}
    \dim F_p^a(W)
    \geq
    N-\operatorname{length}(\mathcal O_{X,p}/\mathfrak m_p^a)
    =N-\binom{n+a-1}{n}.
  \end{equation}
  If $N\geq D_q$, the right-hand side is nonnegative for
  $1\leq a\leq q+1$.  Summing \eqref{eq:taylor-codimension} in this range and
  using the hockey-stick identity gives
  \begin{align*}
    T_p(W)
    &\geq \sum_{a=1}^{q+1}\dim F_p^a(W)\\
    &\geq (q+1)N-\sum_{a=1}^{q+1}\binom{n+a-1}{n}\\
    &=(q+1)N-\binom{n+q+1}{n+1}.
  \end{align*}
  This proves \eqref{eq:capacity-strict}.

  Now suppose $N=D_q$.  The same calculation gives the opposite inequality to
  \eqref{eq:capacity-equality-hypothesis}; hence equality holds at every step.
  In particular,
  \[
    \dim F_p^{q+1}(W)=N-D_q=0.
  \]
  The order-$q$ evaluation map is therefore injective.  Its target has
  dimension
  $\operatorname{length}(\mathcal O_{X,p}/\mathfrak m_p^{q+1})=D_q=N$, so it
  is an isomorphism.
\end{proof}

\begin{theorem}
\label{thm:general-line-bundle-bound}
Let $X$ be a smooth irreducible projective variety of dimension $n$, let $A$
be an ample line bundle, and assume that $T_X$ is $\mu_A$-semistable.  Put
\[
  \tau_A(X):=(-K_X)\cdot A^{n-1}.
\]
Assume $\tau_A(X)>0$.  Let $M$ be a line bundle with
$N:=h^0(X,M)>0$, and set
\[
  \lambda_A(M):=\frac{M\cdot A^{n-1}}{\tau_A(X)}.
\]
Then $\lambda_A(M)\geq0$ and
\begin{equation}\label{eq:general-line-bundle-bound}
  h^0(X,M)
  \leq
  \binom{n+\left\lceil (n+1)\lambda_A(M)\right\rceil}{n}.
\end{equation}
Consequently, for any line bundle $L$ and any integer $m\geq1$,
\begin{equation}\label{eq:general-mL-bound}
  h^0(X,mL)
  \leq
  \binom{n+
  \left\lceil
  (n+1)m\,\frac{L\cdot A^{n-1}}{(-K_X)\cdot A^{n-1}}
  \right\rceil}{n}
\end{equation}
whenever $h^0(X,mL)>0.$
\end{theorem}

\begin{proof}
If $N>0$, a nonzero section of $M$ defines an effective divisor linearly
equivalent to $M$.  Since $A$ is ample, $M\cdot A^{n-1}\geq0$, so
$\lambda_A(M)\geq0$.

By Lemma~\ref{lem:symmetric-powers-semistable}, the dual of a semistable
vector bundle is semistable.  Hence $\Omega_X^1$ is $\mu_A$-semistable.  Also
\[
  \mu_A(\Omega_X^1)
  =
  \frac{K_X\cdot A^{n-1}}{n}
  =
  -\frac{\tau_A(X)}{n}.
\]
Apply Lemma~\ref{lem:jet-filtration} to
$W=H^0(X,M)$.  For a general $p\in X$,
\begin{equation}\label{eq:general-slope-jet-upper}
  T_p(W)\leq n\lambda_A(M)N.
\end{equation}
Put
\[
  \lambda:=\lambda_A(M),\qquad
  q:=\left\lceil (n+1)\lambda\right\rceil,\qquad
  D_q:=\binom{n+q}{n}.
\]
Suppose, for contradiction, that $N>D_q$.  Lemma~\ref{lem:taylor-capacity}
gives
\[
  T_p(W)\geq
  (q+1)N-\binom{n+q+1}{n+1}.
\]
Using
\[
  \binom{n+q+1}{n+1}
  =
  \frac{n+q+1}{n+1}D_q
\]
and $q\geq(n+1)\lambda$, we obtain
\begin{align*}
 &(q+1)N-\binom{n+q+1}{n+1}-n\lambda N \\
 &\quad=
 (q+1-n\lambda)(N-D_q)
 +\frac{n(q-(n+1)\lambda)}{n+1}D_q
 >0.
\end{align*}
This contradicts \eqref{eq:general-slope-jet-upper}.  Therefore
$N\leq D_q$, proving \eqref{eq:general-line-bundle-bound}.  Applying this to
$M=mL$ gives \eqref{eq:general-mL-bound}.
\end{proof}

\begin{corollary}\label{cor:general-line-bundle-fano-comparison}
In the setting of Theorem~\ref{thm:general-line-bundle-bound}, if a line
bundle $L$ satisfies
\[
  L\cdot A^{n-1}\leq (-K_X)\cdot A^{n-1},
\]
then, for every $m\geq1$,
\[
  h^0(X,mL)\leq\binom{n+m(n+1)}{n}.
\]
If $L$ is pseudoeffective, then
\[
  \operatorname{vol}(L)
  \leq
  (n+1)^n
  \left(
  \frac{L\cdot A^{n-1}}{(-K_X)\cdot A^{n-1}}
  \right)^n.
\]
\end{corollary}

\begin{proof}
Fix $m\geq1$.  If $h^0(X,mL)=0$, the finite-level assertion is trivial.
Otherwise Theorem~\ref{thm:general-line-bundle-bound}, applied to $M=mL$,
also gives
\[
  \alpha:=
  \frac{L\cdot A^{n-1}}{(-K_X)\cdot A^{n-1}}
  \geq0.
\]
Since $\alpha\leq1$, one has
$\lceil (n+1)m\alpha\rceil\leq(n+1)m$, and the finite-level bound follows
from \eqref{eq:general-mL-bound}.  The volume estimate follows by dividing
the same bound by $m^n/n!$ and letting $m\to\infty$.
\end{proof}

\begin{proof}[Proof of Theorem~\ref{thm:semistable-tangent-quantized-bound}]
  Fix an arbitrary integer $m\geq1$, and set
  \[
    L=mH,
    \qquad
    W=H^0(X,L),
    \qquad
    N=h^0(X,L),
    \qquad
    V=H^n.
  \]
  The assertion is immediate if $N=0$, so assume $N>0$.
  Lemma~\ref{lem:symmetric-powers-semistable} implies that the dual of a
  slope-semistable vector bundle is slope-semistable.  Hence the assumed
  $\mu_H$-semistability of $T_X$ implies that of $\Omega_X^1$, and
  \[
    \mu_H(\Omega_X^1)
    =\frac{K_X\cdot H^{n-1}}{n}
    =-\frac{V}{n},
    \qquad
    \deg_H(L)=mV.
  \]
  Lemma~\ref{lem:jet-filtration} therefore supplies a nonempty open set
  $U_m\subseteq X$ such that, for $p\in U_m$,
  \begin{equation}\label{eq:fano-jet-upper-bound}
    T_p(W)\leq mnN.
  \end{equation}

  Put
  \[
    q:=m(n+1),
    \qquad
    D:=\binom{n+q}{n}.
  \]
  Notice the identity
  \begin{equation}\label{eq:binomial-m-plus-one}
    \binom{n+q+1}{n+1}
    =\frac{n+q+1}{n+1}\binom{n+q}{n}
    =(m+1)D.
  \end{equation}
  If $N>D$, Lemma~\ref{lem:taylor-capacity},
  \eqref{eq:binomial-m-plus-one}, and $q+1-mn=m+1$ give
  \begin{align*}
    T_p(W)
    &\geq(q+1)N-(m+1)D\\
    &=mnN+(m+1)(N-D)
    >mnN,
  \end{align*}
  contradicting \eqref{eq:fano-jet-upper-bound}.  Thus
  \[
    h^0(X,mH)=N\leq D=\binom{n+m(n+1)}{n}.
  \]

  Suppose now that $N=D$.  Lemma~\ref{lem:taylor-capacity} and
  \eqref{eq:binomial-m-plus-one} yield
  \[
    T_p(W)
    \geq(q+1)D-(m+1)D=mnD.
  \]
  Together with \eqref{eq:fano-jet-upper-bound}, this is an equality.  The
  equality clause of Lemma~\ref{lem:taylor-capacity} shows that $|mH|$
  generates $q=m(n+1)$ jets at $p$.  If $s(mH,p)$ denotes the maximal order of
  jets generated at $p$, then
  \[
    m(n+1)=q\leq s(mH,p)\leq m\varepsilon(H,p).
  \]
  For completeness, the second inequality follows directly from the curve
  definition of the Seshadri constant.  If $|A|$ generates $s$-jets at $p$,
  then, for every irreducible curve $C\ni p$, one may prescribe an order-$s$
  initial form that does not vanish on the projectivized tangent cone of $C$.
  The resulting divisor $D\in|A|$ does not contain $C$ and satisfies
  \[
    A\cdot C=D\cdot C\geq s\,\operatorname{mult}_pC.
  \]
  Taking the infimum over $C$ proves $s(A,p)\leq\varepsilon(A,p)$; applying
  this to $A=mH$ and using
  $\varepsilon(mH,p)=m\varepsilon(H,p)$ gives the displayed inequality.
  Therefore $\varepsilon(-K_X,p)=\varepsilon(H,p)\geq n+1$.  By the
  Bauer--Szemberg characterization of projective space
  \cite[Theorem~1.7]{BauerSzemberg2009}, a smooth Fano $n$-fold different from
  $\mathbb P^n$ has $\varepsilon(-K_X,x)\leq n$ at every point.  Hence
  $X\simeq\mathbb P^n$.

  Finally we treat the case where $(-K_X)^n=(-K_{\mathbb P^n})^n$.
  The proof below was overlooked in the previous version of our manuscript and was
  kindly provided to us by Hanfeng Chu.  The case $n=1$ is immediate, so assume
  $n\geq2$.  For every sufficiently large $m$, let $U_m$ be the nonempty open
  set supplied by Lemma~\ref{lem:jet-filtration} for
  $W_m:=H^0(X,-mK_X)$, and put $N_m:=\dim W_m$.  Choose a very general point
  $p\in\bigcap_{m\gg1}U_m$.  Thus
  \[
    T_p(W_m)\leq mnN_m
  \]
  for every sufficiently large $m$.  Let $\tilde X\xrightarrow{\pi}X$ be the
  blowup of $p$, with exceptional divisor $E$.  Then, for $a\geq0$,
  \[
    H^0(X,-mK_X\otimes\mathfrak m_p^a)
    =H^0(\tilde X,\pi^*(-mK_X)-aE).
  \]
  In the notation of \cite[Sections~2.6]{BlumJonsson2020}, the corresponding finite-level expected
  vanishing order is
  \[
    S_m(E):=\frac{1}{mN_m}T_p(W_m)
    =\frac{1}{m h^0(X,-mK_X)}
      \sum_{a\geq1}h^0(\tilde X,\pi^*(-mK_X)-aE).
  \]
  Hence $S_m(E)\leq n$.  Letting $m\to\infty$ and using
  \cite[Lemma 2.9]{BlumJonsson2020}, we obtain
  \[
    S(E):=\frac{1}{\mathrm{vol}(H)}
      \int_0^\infty\mathrm{vol}(\pi^*(-K_X)-xE)\,dx\leq n.
  \]
  On the other hand, Fujita's point-blowup estimate
  \cite[Theorem~2.3(1)]{Fujita2018} gives
  \[
    \begin{aligned}
      S(E)
      &\geq\frac{1}{\mathrm{vol}(-K_X)}
        \int_0^{\mathrm{vol}(-K_X)^{1/n}}
        \bigl(\mathrm{vol}(-K_X)-x^n\bigr)\,dx \\
      &=\frac{n}{n+1}\mathrm{vol}(-K_X)^{1/n}=n,
    \end{aligned}
  \]
  where we used $\mathrm{vol}(-K_X)=(-K_X)^n=(n+1)^n$ in the last
  equality.  Thus equality $S(E)=n$ holds, forcing that
  \[
    \mathrm{vol}(\pi^*(-K_X)-xE)=(n+1)^n-x^n
    \qquad\text{for every }x\in[0,n+1].
  \]
  Fujita's characterization
  \cite[Theorem~2.3(2)]{Fujita2018} now gives
  $\varepsilon(-K_X,p)=n+1$, and hence
  \cite[Theorem~1.7]{BauerSzemberg2009} yields
  $X\simeq\mathbb P^n$.
\end{proof}

\begin{corollary}[=Theorem \ref{thm:main-intro}]\label{cor:K-semistable-quantized-bound}
  Let $X$ be a smooth complex $K$-semistable Fano variety.  Then the conclusion
  of Theorem~\ref{thm:semistable-tangent-quantized-bound} holds for every
  $m\geq1$.
\end{corollary}

\begin{proof}
  Li's theorem \cite[Theorem~1.3]{Li2021}, extending Tian's
  K\"ahler--Einstein/\allowbreak Ricci-lower-bound results on the canonical extension
  \cite[Theorems~0.1--0.2]{Tian1992}, says that the tangent bundle of a
  smooth $K$-semistable Fano variety is $\mu_{-K_X}$-semistable.  Applying
  Theorem~\ref{thm:semistable-tangent-quantized-bound}, we conclude.
\end{proof}
\begin{remark}
Finally we point our that the same argument extends to $K$-semistable $\mathbb Q$-Fano
varieties after the standard reflexive modifications.  If
$j:X_{\mathrm{reg}}\hookrightarrow X$ and
\[
 \mathcal L_m:=\mathcal O_X(-mK_X)
 :=\bigl(\mathcal O_X(-K_X)^{\otimes m}\bigr)^{**},
\]
one constructs the principal-parts filtration on $X_{\mathrm{reg}}$, where
$\mathcal L_m$ is locally free, and then replaces $\Omega_X^1$,
symmetric powers, tensor products, and the graded images by their reflexive
counterparts on $X$.  Since a normal variety is regular in codimension one,
the torsion and cokernels introduced by reflexive extension are supported in
codimension at least two and do not affect ranks, determinants, first Chern
classes, or slopes.  Semistability of the reflexive tangent sheaf is supplied
by \cite[Theorem~6(i) and Remark~4]{DruelGuenanciaPaun2024}; the required
semistability of reflexive symmetric powers follows by restriction to a
general complete-intersection curve and the characteristic-zero tensor
product theorem, as in
\cite{Flenner1984,RamananRamanathan1984,HuybrechtsLehn2010}.

The jet counting argument itself is unchanged, since it is applied at a
general smooth point.  Consequently the same computation gives
\[
 h^0\bigl(X,\mathcal O_X(-mK_X)\bigr)
 \leq \binom{n+m(n+1)}{n}
\]
for every $m\geq1$, including values for which $-mK_X$ is not Cartier.  In
the equality case, a section of the reflexive anticanonical power defines an
effective $\mathbb Q$-Cartier Weil divisor; the jet-isomorphism at the chosen
smooth point yields $\varepsilon(-K_X,p)\geq n+1$, and the singular
Seshadri characterization of projective space
\cite[Theorem~2]{LiuZhuang2018} gives $X\simeq\mathbb P^n$.
\end{remark}

\end{document}